\documentclass{ws-ijnt}
\usepackage{amsmath}
\usepackage{amssymb}
\usepackage{pstricks}
\usepackage{pstricks-add}
\usepackage[OT2,T1]{fontenc}
\usepackage{listings}

\newcommand{\s}[1]{\mbox{\sf #1}}
\newcommand{\Q}{\mathbb Q}
\newcommand{\Z}{\mathbb Z}
\newcommand{\Gal}{\text{Gal}}
\DeclareSymbolFont{cyrletters}{OT2}{wncyr}{m}{n}
\DeclareMathSymbol{\Sha}{\mathalpha}{cyrletters}{"58}

\begin{document}

\markboth{Cheol-Min Park, Sun Woo Park}{Minimal degrees of algebraic numbers with respect to primitive elements}

%
\catchline{}{}{}{}{}
%

\title{Minimal degrees of algebraic numbers with respect to primitive elements}

\author{Cheol-Min Park}

\address{Division of Advanced Researches for Industrial Mathematics, National Institute for Mathematical Sciences\\
70, Yuseong-daero 1689 beon-gil, Yuseong-gu, Daejeon, Republic of Korea\\
\email{mpcm@nims.re.kr} }

\author{Sun Woo Park}

\address{Division of Advanced Researches for Industrial Mathematics, National Institute for Mathematical Sciences\\
70, Yuseong-daero 1689 beon-gil, Yuseong-gu, Daejeon, Republic of Korea\\
\email{spark483@nims.re.kr} }

\maketitle

\begin{history}
\received{(Day Month Year)}
\accepted{(Day Month Year)}
\end{history}

\begin{abstract}
Given a number field $L$, we define the degree of an algebraic number $v \in L$ with respect to a choice of a primitive element of $L$. We propose the question of computing the minimal degrees of algebraic numbers in $L$, and examine these values in degree $4$ Galois extensions over $\Q$ and triquadratic number fields. We show that computing minimal degrees of non-rational elements in triquadratic number fields is closely related to solving classical Diophantine problems such as congruent number problem as well as understanding various arithmetic properties of elliptic curves.
\end{abstract}



\section{Introduction}
Let $L$ be an algebraic number field of degree $n$. Then there exists a primitive element $\alpha \in L$ such that $\Q(\alpha) = L$ and the elements $\{1, \alpha, \cdots, \alpha^{n-1}\}$ generate $L$ as a $\Q$-vector space. Every element $v \in L - \Q$ can be uniquely written in the form
\begin{equation*}\label{std_rep}
v =c_{m-1}\alpha^{m-1}+c_{m-2}\alpha^{m-2}+\cdots+c_{1}\alpha+c_0
\end{equation*}
with $c_0,c_1,...,c_{m-1}\in\Q$ for some $m\le n$.
We encode $v$ by the polynomial $f(x) = \sum_{k=0}^{m-1} c_k x^k$ of degree $n-1$, which is called
the standard representation of $v$ with respect to the primitive element $\alpha$. The encoding of $v$ depends on the choice of a primitive element in $L$. 

It is a natural question to ask what the minimal degree of encoding polynomials of $v\in L$ is. More precisely, the minimal degree is defined as follows:

\begin{definition}[Degrees of algebraic numbers]
Let $L$ be an algebraic number field of degree $n$, and let $\{1, \alpha, \cdots, \alpha^{n-1}\}$ be a $\Q$-basis of $L$ for some primitive element $\alpha \in L$.  Given any element $v\in L$, we can write $v$  uniquely as 
\begin{equation*}
    v = f(\alpha)
\end{equation*}
for some $f\in\mathbb{Q}[x]$ with $\deg f \leq n-1$. Then the degree of $v$ with respect to $\alpha$, written as $\deg_\alpha(v)$, is the degree of $f(x)$.
\end{definition}

\begin{definition}[Minimal degrees of algebraic numbers]
Given any  element $v$ in a number field $L$, the minimal degree of $v$ is the minimum of the degrees of $v$ with respect to all primitive elements $\alpha$ of $L$ and written as
\begin{equation*}
    \min \deg_L (v) := \min_{\alpha \in \mathcal{A}} \deg_\alpha (v)
\end{equation*}
where $\mathcal{A}$ is the set of all primitive elements in $L$.
\end{definition}

Our initial motivation for computing the minimal degrees of algebraic numbers came from constructing a family of pairing-friendly curves with small $\rho$ values \cite{BW,FST,KSS}. On the other hand, the computation also raises other interesting problems such as finding short representations of algebraic numbers over $\Q$. 

In this paper, we compute the minimal degrees of algebraic numbers in degree $4$ Galois extensions over $\Q$ and triquadratic number fields. As far as we are aware, there are no published results on this problem. 

In order to compute the minimal degree, we first show in Section 2 that a lower bound of the minimal degree of $v$ is given by the degree of the field extension of $L$ over $\Q(v)$. Section 3 and 4 then shows that in degree $4$ Galois extension fields and  index-4 subfields of triquadratic number fields, we can compute the minimal degrees of algebraic numbers by finding primitive elements which provide lower bounds of minimal degrees. In Section 5,  we prove that computing the minimal degrees of some elements in index-2 subfields of triquadratic number fields is equivalent to showing the existence of non $2$-torsion rational points of an associated families of elliptic curves. Afterwards, we discuss how classical arithmetic problems such as congruent number problems are related to computations of minimal degrees. We also make asymptotic statements on the probabilistic distribution of minimal degrees over certain families of triquadratic number fields.

This paper is structured as follows. We compute the lower bounds of minimal degrees in Section 2. We compute the minimal degrees of algebraic numbers in $L$ or $M$ where $L$ are Galois extensions of degree 4 over $\Q$  in Section 3
and $M$ are index-4  subfields of triquadratic number fields in Section 4. In the final section, we compute the minimal degrees of algebraic numbers in
index-2  subfields of triquadratic number fields.

\section*{Acknowledgements}
This work was supported by National Institute for Mathematical Sciences (NIMS) grant funded by the Korean government (MSIT) (B20810000). We thank the reviewers for giving helpful and constructive comments and advice.

\section{Lower bounds of minimal degrees of algebraic numbers}

We start with the following proposition, which shows that a lower bound of the minimal degree of $v$ is given by the degree of the field extension of $L$ over $\Q(v)$.

\begin{proposition}\label{main_prop} Let $L$ be a number field. Given any irrational number $v\in L$, the minimal degree of $v$ satisfies the following inequality:
\begin{equation}\label{eq}
\min \deg_L v \ge [L:\mathbb{Q}(v)].
\end{equation}
\end{proposition}
\begin{proof}
Let $K = \Q(v)$. Pick a primitive element $\alpha$ of $L$ such that $\deg_{\alpha}(v) = m$. Then $v$ can be written as
$$
v=f(\alpha)
$$
for some $f\in\mathbb{Q}[x]$ with $\deg f=m$. Let $g(x) := f(x) - v$ be a polynomial over $K$. Because $L = K(\alpha)$ and $\alpha$ is a root of $g$, it follows that
\begin{equation*}
    [L:K] \leq deg(g) = deg(f) = m.
\end{equation*}

\end{proof}
\begin{remark}\label{rmk1} The lower bound is trivially equal to the minimal degree in some cases. 
\begin{itemize}
\item 
Consider the case when $v$ is a primitive element of $L$. Then we have $\deg_v(v)=1$. By Proposition $\ref{main_prop}$, it follows that 
\begin{equation*}
   \min \deg_L (v) = [L:\mathbb{Q}(v)] = 1.
\end{equation*}
In particular, if $L$ is a number field such that $[L:\Q] = p$ for some prime $p$. Because $L$ has no proper non-trivial subfields, any irrational $v \in L$ is a primitive element of $L$. Hence for any $v \in L$, $$\min \deg_L(v) = [L:\Q(v)]. $$
\item 
When $v$ is a rational number in $L$, then $\deg_{\alpha}(v)=0$ for any primitive elements of $L$. Therefore, Proposition $\ref{main_prop}$ does not hold in this case. However, if we consider the indices of subfields up to modulo extension degree of $L$ over $\Q$, the equality in Proposition $\ref{main_prop}$ also holds for rational numbers because 
 \begin{equation*}
   [L:\mathbb{Q}(v)] = [L:\mathbb{Q}] \equiv 0\pmod{[L:\mathbb{Q}]}.
\end{equation*}
\end{itemize}
\end{remark}

\section{Minimal degrees in Galois extensions of degree $4$ over $\Q$.
}
In this section, we show that the equality in Proposition \ref{main_prop} also holds for any degree $4$ Galois extensions over $\Q$.

\begin{theorem}\label{thm_4gal}
Let $L/\Q$ be a Galois extension with $[L:\Q]=4$. Then for any irrational number $v\in L$, there exists a primitive element $\alpha$ in $L$ such that 
$$
\deg_\alpha(v)=[L:\Q(v)] 
$$
\end{theorem}
\begin{proof}
The case where $\Q(v)=L$ is trivial by Remark \ref{rmk1}. It is enough to consider the case where $[L:\Q(v)]=2$. Since $\Gal(L/\Q)$ is isomorphic to $\Z/4Z$ or $\Z/2\Z \times \Z/2\Z$, we prove Theorem \ref{thm_4gal} by dividing into two cases. \vskip 2mm
(Case 1) $\Gal(L/\Q)\cong \Z/4Z$: \\
Let $L=\Q(\alpha)$ and $f(x)=x^2+f_1(v)x+f_0(v)$ be the minimal polynomial of $\alpha$ over $\Q(v)$ for some $f_i(x)\in\Q[x]$. Then we have
\begin{eqnarray*}
f(x)&=&x^2+f_1(v)x+f_0(v) \\
    &=& (x+f_1(v)/2)^2+f_0(v)-f_1(v)^2/4
\end{eqnarray*}
If $f_0(v)-f_1(v)^2/4\in\Q$, then $\alpha+f_1(v)/2$ has a minimal polynomial of degree 2 over $\Q$. Since $L$ has a unique subfield
$\Q(v)$ such that $[\Q(v):\Q]=2$, this implies that $\alpha+f_1(v)/2\in\Q(v).$ This is impossible because  $\alpha$ is a primitive element of $L$. Thus, $$f_0(v)-f_1(v)^2/4\notin \Q$$
Let $f_0(v)-f_1(v)^2/4=a_0+a_1v$ for $a_i\in\Q$. Then we have
$$
v=-\frac{1}{a_1}((\alpha+f_1(v)/2)^2+a_0)
$$
and so $\deg_\alpha(v)=2$.
\vskip 2mm
(Case 2) $\Gal(L/\Q)\cong \Z/2\Z \times \Z/2\Z$: \\
Without loss of generality, we may assume $L=\Q(\sqrt{a},\sqrt{b})$ where $a<b$ are square-free integers.
Note that $L$ has 3 subfields of index 2: $\Q(\sqrt{a}),\Q(\sqrt{b}), \Q(\sqrt{ab})$. Since
$$
\deg_\alpha(v)=\deg_\alpha(c_1v+c_0)
$$
for all $c_0,c_1\in\Q$ with $c_1\neq 0$, we can assume that $v$ is one of $\sqrt{a}, \sqrt{b}, \sqrt{ab}$. If $v=\sqrt{a}$, we take $\alpha=\sqrt{b}+ \sqrt{ab}$. Then we have
\begin{eqnarray*}
\alpha^2&=&2b\sqrt{a}+(b+ab)\\
        &=&2b\cdot v+(b+ab) 
\end{eqnarray*}
Therefore, we have 
$$
v=\frac{1}{2b}\alpha^2-\frac{a+1}{2}
$$
and so $\deg_\alpha(v)=2$. In other cases, we can prove the theorem in an analogous manner. 
\end{proof}

\section{Minimal degrees in index $4$ subfields of triquadratic number fields}
In this section, we show that the minimal degrees of elements in  index $4$ subfields of triquadratic number fields are equal to the lower bounds of minimal degrees in Proposition \ref{main_prop}.

\begin{theorem}
Let $L = \Q(\sqrt{A},\sqrt{B},\sqrt{C})$ where $A,B,C,$ and $ABC$ are distinct square-free non-zero integers. Fix an element $v \in L$ with $[L:\Q(v)] = 4$. Then there exists a primitive element $\alpha$ in $L$ such that 
$$
\deg_\alpha(v)=[L:\Q(v)].
$$
\end{theorem}
\begin{proof}
It suffices to show that $\min \deg_L (\sqrt{A}) = 4$. Pick 
\begin{equation*}
    \alpha = a \sqrt{B} + b \sqrt{C} + c \sqrt{AB} + d \sqrt{AC}
\end{equation*}
for some $a, b, c, d \in \Q^\times$. This ensures that $\alpha$ is a primitive element of $L$. Then
\begin{equation*}
\alpha^2=X+Y\sqrt{A}+Z\sqrt{BC}+W\sqrt{ABC},
\end{equation*}
where 
\begin{eqnarray*}
X&=& Ba^2+Cb^2+ABc^2+ACd^2,\\
Y&=& 2Bac+2Cbd, \\
Z&=& 2ab+2Acd,\\
W&=& 2bc+2ad.
\end{eqnarray*}
Observe that
\begin{eqnarray*}
    \alpha^4 &=&  (X^2+AY^2+BCZ^2+ABCW^2)+(2XY + 2BCZW)\sqrt{A}+(2XZ+2AYW)\sqrt{BC}\\
             &&+(2YZ+2XW)\sqrt{ABC}.
\end{eqnarray*}
In order for a $\Q$-linear span of $\{1,\alpha^2,\alpha^4\}$ to contain $\sqrt{A}$, we require that the ratio of coefficients of $\sqrt{BC}$ and $\sqrt{ABC}$ in $\alpha^2$ and $\alpha^4$ be the same. In other words,
\begin{equation}\label{eq_coeff}
\frac{Z}{W}=\frac{2XZ + 2AYW}{2YZ + 2XW}
\end{equation}
By solving  (\ref{eq_coeff}), we have
\begin{equation*}
2Y(Z^2 - AW^2) = 0
\end{equation*}
Since $Z, W\in\Q$ and $A$ is square-free, we have $Y=0$. Therefore, we can rewrite $\alpha^2$ and $\alpha^4$ as:
\begin{eqnarray*}
\alpha^2&=& X+Z\sqrt{BC}+W\sqrt{ABC}\\
\alpha^4&=& (X^2+BCZ^2+ABCW^2)+2BCZW\sqrt{A}+2XZ\sqrt{BC}+2XW\sqrt{ABC}
\end{eqnarray*}
This implies that
$$
\alpha^4-2X\alpha^2=(-X^2+BCZ^2+ABCW^2)+2BCZW\sqrt{A}
$$
As long as $Z, W\neq 0$, we have
$$
\sqrt{A}=\frac{1}{2BCZW}\{\alpha^4-2X\alpha^2-(-X^2+BCZ^2 + ABCW^2)\}
$$
Hence we can prove the Theorem if we find $a,b,c,d \in \Q^\times$ satisfying the following 3 conditions: 
$$
\begin{array}{lllll}
\s{C1} & : & 2Bac + 2Cbd&=0   &(\Rightarrow Y=0)\\
\s{C2} & : & 2ab+2Acd&\neq 0&(\Rightarrow Z\neq0)\\
\s{C3} & : & 2bc+2ad&\neq 0&(\Rightarrow W\neq0)
\end{array}
$$
There are infinitely many solutions which satisfy the above system of equations. In particular, $(a, b, c, d) = (1, B, -C, 1)$ satisfies the condition \s{C1}, \s{C2}, and \s{C3}. Condition \s{C2} shows that $B \neq AC$, which is true because $L$ is a triquadratic extension. Condition \s{C3}, which suggests that $2 - 2BC \neq 0$, holds because $B, C$ are distinct square-free integers. Thus, setting the primitive element $\alpha = \sqrt{B} + B\sqrt{C} - C\sqrt{AB} + \sqrt{AC}$, we obtain $\deg_\alpha (\sqrt{A}) = 4$.
\end{proof}

\begin{example}
Let $L = \Q(\sqrt{2}, \sqrt{3}, \sqrt{5})$. Table \ref{ex1} shows a list of primitive elements $\alpha$ of $L$ and polynomials of deg 4 in $\alpha$ which represent $\sqrt{2}, \sqrt{3},\sqrt{5},\sqrt{6},\sqrt{10},\sqrt{15}$ and $\sqrt{30}$.

\begin{table}[h]
\centering
\begin{tabular}{|c|c|c|}\hline
Elements & $\alpha$ & Polynomials of deg 4 \\\hline
$\sqrt{2}$ & $\sqrt{3} + 3\sqrt{5} - 5\sqrt{6} + \sqrt{10}$ & $\frac{1}{11760} \left( \alpha^4 - 416 \alpha^2 +16804\right)$ \\\hline
$\sqrt{3}$ & $\sqrt{2} + 2\sqrt{5} - 5\sqrt{6} + \sqrt{15}$  & $\frac{1}{9360} \left( \alpha^4 - 374 \alpha^2 + 18489 \right)$ \\\hline
$\sqrt{5}$& $\sqrt{2} + 2\sqrt{3} - 3\sqrt{10} + \sqrt{15}$  & $\frac{1}{3120} \left( \alpha^4 -238\alpha^2+7105\right)$ \\\hline
$\sqrt{6}$& $\sqrt{2} -10\sqrt{3} + 2\sqrt{5} + \sqrt{30}$  & $\frac{1}{20160} \left( \alpha^4 -704\alpha^2 + 73104 \right)$ \\\hline
$\sqrt{10}$& $\sqrt{2} + 2\sqrt{3} - 6\sqrt{5} + \sqrt{30}$  & $\frac{1}{6720} \left( \alpha^4 -448\alpha^2 + 25360 \right)$ \\\hline
$\sqrt{15}$& $\sqrt{2} + 2\sqrt{3} + 3\sqrt{5} -3\sqrt{30}$  & $\frac{1}{10320} \left( \alpha^4 -658\alpha^2 + 54865 \right)$ \\\hline
$\sqrt{30}$& $\sqrt{2} + 2\sqrt{3} + 3\sqrt{10} - 6\sqrt{15}$  & $\frac{1}{21120} \left( \alpha^4 -1288\alpha^2 + 210880 \right)$ \\\hline
\end{tabular}
\caption{\label{ex1} minimal degrees of some elements in $\Q(\sqrt{2}, \sqrt{3}, \sqrt{5})$}
\end{table}

\end{example}

\section{Minimal degrees in index $2$ subfields of triquadratic number fields}

\subsection{Necessary and sufficient conditions for minimal degrees}

In this section, we show that the minimal degrees of elements $v$ in triquadratic number fields $L$ can be strictly greater than $[L:\Q(v)]$. In fact, we prove that the minimal degrees of elements in index $2$ subfields of triquadratic number fields are determined by arithmetic properties of certain families of elliptic curves. 

\begin{theorem}\label{index2}
Let $L = \mathbb{Q}(\sqrt{A}, \sqrt{B}, \sqrt{C})$ be a triquadratic number field where  $A,B,C,$ and $ABC$ are distinct square-free non-zero integers. Then for any non-zero rational number $a$, 
$\min \deg_L (\sqrt{A} + a\sqrt{B}) = 2$ if and only if the rank of the elliptic curve $E: Y^2 = X(X-a^2B)(X-(a^2B-A))$ is at least $1$ or the torsion subgroup of $E$ is isomorphic to $\Z/2\Z \times \Z/6\Z$.
\end{theorem}

We state the following lemma which classifies the torsion subgroups of the  elliptic curve $E: Y^2 = X(X-a^2B)(X-(a^2B-A))$.
\begin{lemma}\label{torsion}
Let $A, B$ be distinct square-free non-zero integers. Given any non-zero rational number $a$, let $E: Y^2 = X(X-a^2B)(X-(a^2B-A))$ be an elliptic curve over $\Q$. Then the torsion subgroup $E_{Tor}(\Q)$ is isomorphic to $\Z/2\Z \times \Z/2\Z$ or $\Z/2\Z \times \Z/6\Z$.
In particular, $E_{Tor}(\Q)$ is isomorphic to $\Z/2\Z \times \Z/6\Z$ if and only if there exist integers $p,q$ satisfying 
\begin{align*}
    -a^2B &= p^4 + 2p^3q \\
    -a^2B + A &= 2pq^3 + q^4
\end{align*}

\end{lemma}
\begin{proof}

Choose a rational number $a = \frac{m}{n}$ such that $(m,n) = 1$. Then the elliptic curve $E:Y^2 = X(X-a^2B)(X-(a^2B-A))$ is isomorphic to $E':Y^2 = X(X-m^2B)(X-(m^2B-n^2A))$. Because $A, B$ are square-free integers, $\pm m^2B$ and $\pm n^2A$ are not squares. Then Lemma \ref{torsion} follows from Main Theorem 1 in \cite{Ono}.
\end{proof}

We now prove Theorem \ref{index2}.
\begin{proof}
Suppose $\alpha$ is a primitive element of $L$ which satisfies $\deg_{\alpha}(\sqrt{A}+a\sqrt{B})=2$. Thus 
\begin{equation}\label{eq_index2}
\sqrt{A}+a\sqrt{B}=a_2\alpha^2 +a_1\alpha +a_0
\end{equation}
for some $a_0,a_1\in \Q$ and non-zero $a_2\in\Q$.
Let $f(x)=x^2 +\frac{a_1}{a_2} x +\frac{a_0-(\sqrt{A}+a\sqrt{B})}{a_2}$. Let $\sigma_1$ be the identity element in  $\Gal(L/\Q)$ and $\sigma_2$ be an element of $\Gal(L/\Q)$ satisfying $$\sigma_2(\sqrt{A})= \sqrt{A}, ~\sigma_2(\sqrt{B})= \sqrt{B},~\sigma_2(\sqrt{C})=- \sqrt{C}. $$
Since $\sigma_1(\alpha)$ and $\sigma_2(\alpha)$ are roots of $f(x)$,  we have
\begin{eqnarray}
\sigma_1(\alpha)+\sigma_2(\alpha) &=& -\frac{a_1}{a_2} \label{eq_sum}\\
\sigma_1(\alpha)\sigma_2(\alpha) &=& \frac{a_0-(\sqrt{A}+a\sqrt{B})}{a_2} \label{eq_pd}
\end{eqnarray}
Let $\alpha=b_0 +b_1 \sqrt{A} +b_2 \sqrt{B}+b_3 \sqrt{C}+b_4 \sqrt{AB}+b_5 \sqrt{AC}+b_6 \sqrt{BC}+b_7 \sqrt{ABC}$ for some $b_i\in\Q$.
By  (\ref{eq_sum}), we have
$$
2(b_0 +b_1 \sqrt{A} +b_2 \sqrt{B}+b_4 \sqrt{AB})\in\Q. 
$$
Hence, $b_1 =b_2=b_4 =0$ and so $$\alpha=b_0 +b_3 \sqrt{C}+b_5 \sqrt{AC}+b_6 \sqrt{BC}+b_7 \sqrt{ABC}.$$ This gives
\begin{eqnarray*}
\sigma_1(\alpha)\sigma_2(\alpha)&=&b_0^2 -(b_3 \sqrt{C}+b_5 \sqrt{AC}+b_6 \sqrt{BC}+b_7 \sqrt{ABC})^2 \\
             &=&-\{ (ABCb_7^2 +BCb_6^2 +ACb_5^2 +Cb_3^2 -b_0^2 )+\sqrt{A}(2BCb_7b_6+ 2Cb_3b_5)\\&&~~~+\sqrt{B}(2ACb_7b_5+2Cb_3b_6)+\sqrt{AB}(2Cb_7b_3+2Cb_5b_6) \}   
\end{eqnarray*}
By  (\ref{eq_pd}), we have the following system of equations:

\begin{equation}\label{system_eq}
\left\{
\begin{array}{rl}
ABCb_7^2 +BCb_6^2 +ACb_5^2 +Cb_3^2 -b_0^2 &= -\frac{a_0}{a_2} \\
2BCb_7b_6+ 2Cb_3b_5 &= \frac{1}{a_2} \\
2ACb_7b_5+2Cb_3b_6 &=  \frac{a}{a_2} \\
2Cb_7b_3+2Cb_5b_6 &= 0 
\end{array} 
\right.
\end{equation}

Note that if $b_5 =0$, then $b_3=0$ or $b_7=0$ by the 4th equation of  (\ref{system_eq}). However, by the 2nd and the 3rd equation of  (\ref{system_eq}), this is impossible. Thus, $b_5\ne0$. Let
\begin{equation}\label{system_eq1.5}
\left\{
\begin{array}{rl}
\tilde{x} &= \frac{b_3}{b_5} \\
\tilde{y} &= \frac{b_6}{b_5} \\
\tilde{z} &= \frac{b_7}{b_5}
\end{array} 
\right.
\end{equation}
Then we obtain the following equations from  (\ref{system_eq}): 
\begin{equation} \label{system_eq2}
\left\{
\begin{array}{rlll}
aB\tilde{y}\tilde{z}+a\tilde{x}&=&A\tilde{z}+\tilde{x}\tilde{y}&\\
\tilde{x}\tilde{z}+\tilde{y}&=&0&  
\end{array}  
\right.
\end{equation}
Eliminating $\tilde{x}$-variables in  (\ref{system_eq2}), we have
\begin{equation}\label{disc}
C: aB\tilde{y}\tilde{z}^2-a\tilde{y}=A\tilde{z}^2-\tilde{y}^2,
\end{equation}
which is a curve in the affine plane $\mathbb{A}^2_\Q = \text{Spec}(\Q[\tilde{y},\tilde{z}])$. 
The projectivization $\tilde{C} \subset \mathbb{P}^2_\Q = \text{Proj}(\Q[\tilde{y}, \tilde{z}, \tilde{w}])$ given by 
\begin{equation} \label{system_eq2.5}
    \tilde{C}: aB\tilde{y}\tilde{z}^2-a\tilde{y}\tilde{w}^2=A\tilde{z}^2\tilde{w}-\tilde{y}^2\tilde{w}
\end{equation}
defines an elliptic curve over $\mathbb{Q}$. Applying a rational change of coordinates
\begin{equation} \label{rational_change}
\left\{
\begin{array}{rl}
    \tilde{y} &= AX \\
    \tilde{z} &= Y \\
    \tilde{w} &= aBX - (a^3B^2 - aAB)Z
\end{array}  
\right.
\end{equation}
we obtain a Weierstrass model $\tilde{E} \subset \mathbb{P}^2_\Q = \text{Proj}(\Q[X,Y,Z])$ given by
\begin{equation*}
    \tilde{E} : Y^2Z = X^3 - (2a^2B - A)X^2Z + a^2B(a^2B - A)XZ^2 =
    X(X-a^2BZ)(X-(a^2B-A)Z),
\end{equation*}
or when $Z = 1$, an affine curve $E \subset \mathbb{A}^2_\Q = \text{Spec}(\Q[X,Y])$ defined as
\begin{equation} \label{ell_min_deg}
    E: Y^2 = X(X-a^2B)(X-(a^2B-A)).
\end{equation}
Lemma \ref{torsion} shows that $E_{Tor}(\mathbb{Q}) = \mathbb{Z}/2\mathbb{Z} \times \mathbb{Z} / 2\mathbb{Z}$ or $\Z/2\Z \times \Z/6\Z$. We note that 
\begin{equation*}
    \tilde{E}[2](\mathbb{Q}) = \left\{ [0: 1: 0], [0: 0: 1], [a^2B: 0: 1], [a^2B - A: 0: 1] \right\}
\end{equation*}
The corresponding rational solutions in terms of $[\tilde{y}:\tilde{z}:\tilde{w}]$ are:
\begin{equation*}
    \tilde{C}[2](\mathbb{Q}) = \left\{ [0: 1: 0], [0: 0: -a^3B^2 + aAB], [a^2AB : 0 : aAB], [A(a^2B-A) : 0 : 0] \right\}
\end{equation*}
Hence, $\Q$-rational solutions of (\ref{disc}) induced from $2$-torsion points of $E$ satisfy $\tilde{z} = 0$. This implies that $b_7 = 0$, so $\alpha$ is not a primitive element of $\mathbb{Q}(\sqrt{A}, \sqrt{B}, \sqrt{C})$. We note that rational points of $E$ which are not $2$-torsion points satisfy $X, Y \neq 0$. Hence, the existence of a primitive element $\alpha$ implies that the elliptic curve $E$ either contains a $3$-torsion point or has rank at least $1$.

To prove the converse of the theorem, it is enough to find $(a_0, a_1, a_2)\in\Q^3$ and a primitive element  $\alpha$ of $L$ satisfying  (\ref{index2}). Suppose the desired elliptic curve $E: Y^2 = X(X-a^2B)(X-(a^2B-A))$ has rank at least $1$, or contains a $3$-torsion point. Let $\tilde{E}$ be the projectivization of $E$. Choose a non 2-torsion point $[X:Y:1] \in \tilde{E}(\mathbb{Q})$. This guarantees that $X \neq a^2B - A$ and $Y \neq 0$.

Let $\tilde{C}$ be the projective curve from (\ref{system_eq2.5}). Using the rational change of coordinates from  (\ref{rational_change}), we obtain a point $[\tilde{y} : \tilde{z} : \tilde{w}] \in \tilde{C}(\Q)$ where
\begin{align*}
    \tilde{y} &= AX \\
    \tilde{z} &= Y \\
    \tilde{w} &= aBX - (a^3B^2 - aAB).
\end{align*}
Because $X \neq a^2B - A$, we have $w \neq 0$. Hence we obtain a point $(\tilde{y}, \tilde{z}) \in C(\Q)$ by de-homogenizing the variables:
\begin{align*}
    \tilde{y} &= \frac{AX}{aBX - (a^3B^2 - aAB)} \\
    \tilde{z} &= \frac{Y}{aBX - (a^3B^2 - aAB)}
\end{align*}
Using the $\Q$-rational point $(\tilde{y}, \tilde{z}) \in C(\Q)$, we can obtain infinitely many solutions of $b_3, b_5, b_6,$ and $b_7$. To see this, recall from (\ref{system_eq1.5}) that $\tilde{y} = b_6/b_5$, $\tilde{z} = b_7/b_5$. Using the equation $\tilde{x}\tilde{z} + \tilde{y} = 0$ from  (\ref{system_eq2}), we can also obtain the value of $\frac{b_3}{b_5}$ because
\begin{equation*}
    \tilde{x} = \frac{b_3}{b_5} = -\frac{\tilde{y}}{\tilde{z}} = \frac{AX}{Y}
\end{equation*}
Note that $\frac{b_3}{b_5}$ is well defined because $Y \neq 0$. Since all of $\tilde{x}, \tilde{y}, \tilde{z}$ are non-zero, none of $b_3, b_5, b_6, b_7$ are zero. This implies that $\alpha$ is a primitive element of $L$. 

We verified that a non 2-torsion rational point $[X:Y:1] \in \tilde{E}(Q)$ gives rise to non-zero solutions $b_3, b_5, b_6, $and $b_7$ satisfying 
\begin{equation} \label{intermediate_sol}
\left\{
\begin{array}{rl}
    \frac{b_3}{b_5} &= \frac{AX}{Y} \\
    \frac{b_6}{b_5} &= \frac{AX}{aBX - (a^3B^2 - aAB)}\\
    \frac{b_7}{b_5} &= \frac{Y}{aBX - (a^3B^2 - aAB)}
\end{array}  
\right.
\end{equation}
Using (\ref{system_eq}) and (\ref{intermediate_sol}), we derive $(a_0, a_1, a_2) \in \Q^3$ satisfying (\ref{index2}). The assumption that $a \neq 0$ implies that the third equation from (\ref{system_eq}) does not vanish to $0$.
\begin{equation} \label{intermediate_sol2}
    2AC b_7 b_5 + 2C b_3 b_6 = -\frac{a}{a_2}
\end{equation}
Substituting the expressions from (\ref{intermediate_sol}), we obtain that the LHS of (\ref{intermediate_sol2}) is
\begin{equation*}
    2AC b_7 b_5 + 2C b_3 b_6 = 2C b_5^2 \left( A \frac{b_7}{b_5} + \frac{b_3}{b_5} \frac{b_6}{b_5} \right) = 2C b_5^2 \frac{AY^2 + A^2X^2}{Y(aBX - (a^3B^2 - aAB)}.
\end{equation*}
Because $A$ is square-free and $C, b_5$ are non-zero, the LHS of the third equation is nonzero. We can hence obtain the value of $a_2$ such that 
$$a_2=-\frac{a}{2AC b_7 b_5 + 2C b_3 b_6}$$ 
By choosing a random $b_0\in\Q$ and using the first equation from  (\ref{system_eq}), we can obtain the value of $a_0$ such that 
\begin{equation*}
  a_0=a_2\cdot(  ABC b_7^2 + BC b_6^2 + AC b_5^2 + Cb_3^2 - b_0^2) 
\end{equation*}
Using  (\ref{eq_sum}), we can also determine $a_1 = -2b_0 a_2$. Since these choices of $(a_0, a_1, a_2)$ and a primitive element  $\alpha$ satisfy  (\ref{eq_sum}) and  (\ref{eq_pd}), they also satisfy  (\ref{eq_index2}).
\end{proof}

%

Theorem \ref{index2} relates the problem of computing minimal degrees of elements in $L$ to the problem of understanding arithmetic properties of families of elliptic curves. As an immediate corollary, we show that computing the minimal degree of $\sqrt{A} + a\sqrt{B}$ can be considered as a generalization of the congruent number problem \cite{Tunnell,Koblitz}.
\begin{corollary}
Let $L = \mathbb{Q}(\sqrt{B}, \sqrt{2B}, \sqrt{C})$ be a triquadratic number field for any distinct square-free non-zero integers $B$ and $C \neq 2, 2B$. Then $\min \deg_L (\sqrt{B} + \sqrt{2B}) = 2$ if and only if $B$ is a congruent number.
\end{corollary}
\begin{proof}
Consider the elliptic curve $E: Y^2 = X(X-B)(X+B)$ over $\Q$. If there exist integers $p$ and $q$ which satisfy
\begin{align*}
    B &= p^4 + 2p^3 q \\
    -B &= 2pq^3 + q^4
\end{align*}
then we have
\begin{equation}
    2B = p^4 + 2p^3q - 2pq^3 - q^4 = -(p-q)(p+q)^3
\end{equation}
Because $B$ is square-free, the left hand side of the above equation is divisible by $2$ but not divisible by $8$. However, the right hand side of the equation is either odd (when $p$ is odd and $q$ is even, and vice versa) or divisible by $16$ (when both $p,q$ are odd or even), a contradiction. By Lemma \ref{torsion}, $E_{Tor}(\Q)$ is isomorphic to $\Z/2\Z \times \Z/2\Z$.

Theorem \ref{index2} implies that $\min \deg_L (\sqrt{B} + \sqrt{2B}) = 2$ if and only if the rank of the elliptic curve $E$ is at least $1$. This is equivalent to the statement that $B$ is a congruent number \cite{Koblitz, Tunnell}.
\end{proof}
%

\begin{example}
Let $L$ be the triquadratic number field $\mathbb{Q}(\sqrt{2}, \sqrt{3}, \sqrt{5})$. Then we have $$\deg_{\alpha}(\sqrt{2}+\sqrt{3})>2$$ for every primitive element $\alpha$ in $L$ because the rank of the elliptic curve $E: Y^2 = X^3 - 4X^2 + 3X$ is equal to $0$, and the torsion subgroup of $E$ is isomorphic to $\Z/2\Z \times \Z/2\Z$. On the other hand, there exists a primitive element $\alpha$ such that $$\deg_{\alpha}(\sqrt{2}+2\sqrt{3}) = 2,$$ because $(X,Y) = (8, 8)$ is a non-torsion $\mathbb{Q}$-rational point of $E: Y^2 = X^3 - 22X^2 + 120X$. Associated to the rational point is the primitive element $\alpha = 1 - 2 \sqrt{5} + \sqrt{10} - \frac{4}{3} \sqrt{15} - \frac{2}{3} \sqrt{30}$ and we have $$\frac{3}{20}\alpha^2 - \frac{3}{10}\alpha - \frac{207}{20} = \sqrt{2} + 2\sqrt{3}.$$
\end{example}

\begin{example}
Let $L$ be the triquadratic number field $\Q(\sqrt{5}, \sqrt{7}, \sqrt{11})$. Then there exists a primitive element $\alpha$ such that $$\deg_{\alpha}(\sqrt{11} + 5\sqrt{35}) = 2.$$ $(X,Y)  = (900, 900)$ is a 3-torsion $\Q$-rational point of $E: Y^2 = X(X-5^2 \times 35)(X-(5^2 \times 35-11)) = X(X-875)(X-864)$. We note that the rank of the elliptic curve $E$ is $0$. Associated to the torsion point is the primitive element $\alpha = 1 - 11\sqrt{5} + \sqrt{55} + \frac{55}{7} \sqrt{7} + \frac{5}{7} \sqrt{77}$ and we have $$-\frac{7}{220} \alpha^2 + \frac{7}{110} \alpha + \frac{7913}{220} = \sqrt{11} + 5\sqrt{35}.$$
\end{example}

\begin{remark}
Let $v = \sqrt{A} + a\sqrt{B}$ be a fixed element in $L = \Q(\sqrt{A}, \sqrt{B}, \sqrt{C})$. Theorem \ref{index2} shows that even if the elliptic curve $E: Y^2 = X(X-a^2B)(X-(a^2B-A))$ has rank $0$, as long as $E_{Tor}(\Q) = \Z/2\Z \times \Z/6\Z$ it is possible to find a primitive element $\alpha$ such that $\deg_\alpha(\sqrt{A} + a\sqrt{B}) = 2$ for any non-zero $a\in\Q$. Lemma \ref{torsion} shows that $E_{Tor}(\Q) = \Z/2\Z \times \Z/6\Z$ if and only if there exist integers $p,q$ satisfying \begin{align*}
    -a^2B &= p^4 + 2p^3q \\
    -a^2B + A &= 2pq^3 + q^4
\end{align*}
Rearranging the above equation gives
\begin{equation*}
    -A = (p+q)^3(p-q)
\end{equation*}
Because $A$ is square-free, we require that $p + q = \pm 1$. Hence, we have
\begin{equation}
    \begin{cases}
    a^2B &= -p^3(\pm2 - p) \\
    A &= 1\pm2p 
    \end{cases}
\end{equation}
If $p$ is a square-free integer such that $-p(\pm 2 - p)$ and $1 \pm 2p$ are both square-free, then there exists a primitive element $\alpha \in \Q(\sqrt{A}, \sqrt{B}, \sqrt{C})$ such that $\deg_\alpha(\sqrt{A} + a\sqrt{B}) = 2$. For example, choosing $p = 5$ and $a = 5$ deduces the previous example. 

\end{remark}

\subsection{Minimal degrees in families of triquadratic number fields}

It is a natural question to calculate the probability that the minimal degree of a given element in $M \subset L = \mathbb{Q}(\sqrt{A}, \sqrt{B}, \sqrt{C})$ is equal to $[L : M]$. We show that the desired probability depends on the choice of a family of tuples of form $(L, M, v)$, where $L = \Q(\sqrt{A}, \sqrt{B}, \sqrt{C})$, $M = \Q(\sqrt{A}, \sqrt{B})$, and $v = \sqrt{A} + a \sqrt{B}$ for some distinct non-zero square-free integers $A, B, C$ and some rational number $a$.

\begin{theorem}
Let $L = \mathbb{Q}(\sqrt{A}, \sqrt{B}, \sqrt{C})$ for any fixed distinct square-free non-zero integers $A, B,$ and $C$. Fix a non-zero rational number $a \in \Q$. Let $S$ be the set of primes including $2$, $\infty$, and all finite places at which the elliptic curve $E: Y^2 = X(X-a^2B)(X-(a^2B-A))$ has bad reduction. Fix an integer 
\begin{equation*}
    D = 8 \prod_{\substack{ p \in S \; \\ \; p \; \text{finite} }} p.
\end{equation*}
Let $M_n$ be the following family of number fields.
    \begin{equation*}
        M_n := \{ L_\gamma = \mathbb{Q}(\sqrt{\gamma A}, \sqrt{\gamma B}, \sqrt{C}) \; | \; \gamma \leq n, \gamma \in \mathbb{N}, \; (\gamma, D) = 1,\; \gamma \; \text{square-free} \} 
    \end{equation*}
    Given any fixed non-zero $a\in\Q$, we define the probability that $L_\gamma$ has an element in a degree $4$ subfield $M_\gamma := \mathbb{Q}(\sqrt{\gamma A}, \sqrt{\gamma B})$ with minimal degree greater than $[L_\gamma:M_\gamma]$ as follows.
    \begin{equation*}
        \mathbb{P}(L_\gamma \in M_n : \min \deg_{L_\gamma} (\sqrt{\gamma A} + a\sqrt{\gamma B}) > 2)
        := \frac{\# \{L_\gamma \in M_n \; | \; \min \deg_{L_\gamma} (\sqrt{\gamma A} + a\sqrt{\gamma B}) > 2\}}{\# M_n}
    \end{equation*}
    Then the lower bound of the probability converges to the following value as $n \to \infty$.
    \begin{equation*}
        \liminf_{n \to \infty} \mathbb{P}(L_\gamma \in M_n : \min \deg_{L_\gamma} (\sqrt{\gamma A} + a\sqrt{\gamma B}) > 2) \geq \frac{1}{\prod_{j = 0}^\infty (1 + 2^{-j})}
    \end{equation*}
\end{theorem}

\begin{proof}
We first note that quadratic twists of $E$ by $\gamma$ is
\begin{equation} \label{quad_ell_min_deg}
    E_\gamma : Y^2 = X(X - a^2B\gamma)(X - (a^2B - A)\gamma)
\end{equation}
Theorem \ref{index2} shows that the above elliptic curve is induced from finding a primitive element $\alpha \in L_\gamma = \mathbb{Q}(\sqrt{\gamma A}, \sqrt{\gamma B}, \sqrt{C})$ such that $\deg_\alpha(\sqrt{\gamma A} + a \sqrt{\gamma B}) = 2$. Let $\mathcal{M}_n$ be the family of quadratic twists of elliptic curves $E$ such that
\begin{equation}
    \mathcal{M}_n := \{ E_\gamma : Y^2 = X(X - a^2B\gamma)(X - (a^2B - A)\gamma) \; | \; \gamma \leq n, \gamma \in \mathbb{N}, (\gamma, D) = 1,\; \gamma \; \text{square-free} \}.
\end{equation}
By Theorem 4.2 in \cite[Chap. X.4]{Silverman}, we can consider the following short exact sequence
\begin{equation*}
    0 \to E_\gamma(\mathbb{Q})/2E_\gamma(\mathbb{Q}) \to \text{Sel}_2(E_\gamma) \to \Sha_{E_\gamma}[2] \to 0
\end{equation*}
where $\text{Sel}_2(E_\gamma)$ is the 2-Selmer group and $\Sha_{E_\gamma}$ is the Tate-Shafarevich group.
Lemma \ref{torsion} implies that $E[2](\Q) = \Z/2\Z \times \Z/2\Z$. Hence, we have
\begin{equation}
    \text{dim}_{\mathbb{F}_2} E_\gamma(\mathbb{Q})/2E_\gamma(\mathbb{Q}) = \text{rank}(E_\gamma) + 2 \leq \text{dim}_{\mathbb{F}_2} \text{Sel}_2(E_\gamma).
\end{equation}
In particular, if $\text{dim}_{\mathbb{F}_2} \text{Sel}_2(E_\gamma) = 2$, then $\text{rank}(E_\gamma) = 0$.
Hence, the map of sets
\begin{equation*} \label{correspondence_ineq}
    \{E_\gamma \in \mathcal{M}_n \; | \; \text{dim}_{\mathbb{F}_2} (\text{Sel}_2 (E_\gamma)) = 2, E_\gamma[3](\Q) = 0 \}  \to 
    \{L_\gamma \in M_n \; | \; \min \deg_{L_\gamma} (\sqrt{\gamma A} + a\sqrt{\gamma B}) > 2\}
\end{equation*}
which sends $E_\gamma$ to $L_\gamma$ is injective.

By Swinnerton-Dyer \cite[Theorem 1]{Swinnerton-Dyer} and Kane \cite[Theorem 3]{Kane}, we have
\begin{equation*}
        \lim_{n \to \infty} \frac{\# \{\gamma \leq n \; | \; \gamma \; \text{square-free}, \; (\gamma, D) = 1,\; \text{dim}_{\mathbb{F}_2} (\text{Sel}_2(E_\gamma)) = 2\}}{\# \{\gamma \leq n \; | \; \gamma \; \text{square-free}, (\gamma, D) = 1\;\}} = \frac{1}{\prod_{j = 0}^\infty (1 + 2^{-j})}
    \end{equation*}
Hence, for any fixed non-zero rational number $a$, we have
\begin{align*}
    & \liminf_{n \to \infty} \mathbb{P}(L_\gamma \in M_n : \min \deg_{L_\gamma} (\sqrt{\gamma A} + a\sqrt{\gamma B}) > 2) \\
    &\geq  \liminf_{n \to \infty} \frac{\# \{\gamma \leq n \; | \; \gamma \; \text{square-free}, (\gamma, D) = 1,\; \text{dim}_{\mathbb{F}_2} (\text{Sel}_2(E_\gamma)) = 2, E_\gamma[3](\Q) = 0\}}{\# \{\gamma \leq n \; | \; \gamma \; \text{square-free}, (\gamma, D) = 1\}} \\
    &\geq \frac{1}{\prod_{j = 0}^\infty (1 + 2^{-j})} - \lim_{n \to \infty} \frac{\# \{\gamma \leq n \; | \; \gamma \; \text{square-free}, (\gamma, D) = 1, \; E_\gamma[3](\Q) \neq 0\}}{\# \{\gamma \leq n \; | \; \gamma \; \text{square-free}, (\gamma, D) = 1\}}
\end{align*}
Note that given a fixed elliptic curve $E$, there only exist finitely many square-free $\gamma$ such that $E_\gamma[3](\Q) \neq 0$. There exists an isomorphism 
\begin{equation*}
    E(\overline{\Q}) \to E_\gamma (\overline{\Q}), \; \; \; (X,Y) \mapsto (\gamma X, \gamma^{\frac{3}{2}}Y).
\end{equation*}
The above isomorphism sends torsion points of order $3$ of $E$ to thoes of $E_\gamma$. Because $E[3](\overline{\Q})$ is finite, there are only finitely many square-free $\gamma$ such that $E_\gamma[3](\Q) \neq 0$. Using this observation, we obtain:
\begin{equation}
    \lim_{n \to \infty} \frac{\# \{\gamma \leq n \; | \; \gamma \; \text{square-free}, (\gamma, D) = 1, \; E_\gamma[3](\Q) \neq 0\}}{\# \{\gamma \leq n \; | \; \gamma \; \text{square-free}, (\gamma, D) = 1\}} = 0.
\end{equation}
We can conclude that
\begin{equation}
    \liminf_{n \to \infty} \mathbb{P}(L_\gamma \in M_n : \min \deg_{L_\gamma} (\sqrt{\gamma A} + a\sqrt{\gamma B}) > 2) \geq \frac{1}{\prod_{j = 0}^\infty (1 + 2^{-j})}.
\end{equation}
\medskip
\end{proof}

The above theorem shows that there exist infinitely many triquadratic number fields $L$ such that the minimal degrees of elements in degree $4$ subfields $M$ of $L$ are strictly greater than $[L:M] = 2$. 
\begin{theorem}\label{inf}
Let $L = \mathbb{Q}(\sqrt{A}, \sqrt{B}, \sqrt{C})$ be a triquadratic number field for any fixed distinct square-free non-zero integers $A, B, C $ and $ABC$. Suppose there exists a pair of non-zero rational numbers $(a, b)$ such that
\begin{equation}\label{triples}
    a^2 - 1 = (B-A) b^2
\end{equation}
Let $M_{a}$ be the family of number fields such that 
\begin{equation*}
    M_{a} := \{ L_B = \mathbb{Q}(\sqrt{A}, \sqrt{B}, \sqrt{C}) \; | \; B \neq a \}
\end{equation*}
Then for every number field $L_B$ in $M_{a}$, we have 
$$\min \deg_{L_B}(\sqrt{A} + a\sqrt{B}) = 2.$$
\end{theorem}
\begin{proof}
The condition that $B \neq a$ guarantees that the elliptic curve $E$ is not singular. By Theorem 1, it suffices to show that any elliptic curve of form $E: Y^2 = X(X-a^2B)(X-(a^2B-A))$ has a non 2-torsion rational point $P$. 
Note that the condition $a^2 - 1 = (B-A)b^2$ implies that
\begin{equation*}
    a^2B - A = (a^2-1)B + (B-A) = (B-A)(Bb^2 + 1)
\end{equation*}
Applying Proposition 1.4 of \cite[Chap. X.1]{Silverman} to the elliptic curve $E: y^2 = x(x-a^2B)(x-(a^2B-A))$ gives
$$
P=\begin{cases}(\frac{a^2}{b^2}(Bb^2 + 1), \frac{a^2}{b^3} (Bb^2+1))  & if~Bb^2 + 1~\mbox{is not a square}  \\
((B-A)a^2, -A(B-A)a^2b) & if~Bb^2 + 1 ~\mbox{is a square} 
\end{cases}
$$

\end{proof}
\begin{corollary}
Given a distinct square-free non-zero integers $A, B$ such that $B-A$ is a square, there exist infinitely many numbers $a\in\Q$ such that $\min \deg_{L_B}(\sqrt{A} + a\sqrt{B}) = 2$ for $L_B\in M_a$.
\end{corollary}
\begin{proof}
Let $B-A=c^2$. Then $(a,b)=(\frac{m^2+n^2}{m^2-n^2},\frac{2mn }{c\cdot(m^2-n^2)})$ satisfies  (\ref{triples})  for an arbitrary pair of integers $m$ and $n$ with $m > n > 0.$ Then the statement of the corollary follows from Theorem \ref{inf}.   
\end{proof}
\begin{example}
Consider the family of triquadratic number fields $\{L_B=\Q(\sqrt{B}, \sqrt{B-2}, \sqrt{C})\}$ for any fixed square-free integer $B \geq 3$, $B-2$, and $C$. Then the element $(a, b) = (3, 2)$ satisfies the equation $a^2 - 1 = 2b^2$. Hence we have $$\min \deg_{L_B}(\sqrt{B-2} + 3\sqrt{B})=2$$ Indeed, computations on Magma \cite{Magma} suggest that the rank of the associated elliptic curve $Y^2 = X(X-9B)(X-(8B+2))$ is always at least $1$. If $4B+1$ is not a square, then $(\frac{9(4B+1)}{4}, \frac{9(4B+1)}{8})$ is a non-torsion rational point of $E$. If $4B+1$ is a square, then $(18, -18(B-2))$ is a non-torsion rational point of $E$.
\end{example}

We finish the paper with the following conjecture, which states that every triquadratic number fields has an element such that $\min \deg_L(v) \neq [L:\Q(v)]$.
\begin{conjecture}\label{conjecture}
Let $L = \mathbb{Q}(\sqrt{A}, \sqrt{B}, \sqrt{C})$ for any distinct square-free non-zero integers $A, B,$ and $C$. Let $M$ be the subfield $\mathbb{Q}(\sqrt{A}, \sqrt{B})$. Then there exists a rational number $a$ such that
\begin{equation*}
    \min \deg_L (\sqrt{A} + a\sqrt{B}) > [L:M]
\end{equation*}
\end{conjecture}

\begin{remark}
Theorem \ref{index2} implies that it suffices to show that given any fixed square-free distinct non-zero integers $A$ and $B$, there exists a rational number $a$ such that the rank of the elliptic curve $E: Y^2 = X(X-a^2B)(X-(a^2B-A))$ is equal to $0$ and the torsion subgroup  $E_{Tor}(\Q)$ is isomorphic to $\Z/2\Z \times \Z/2\Z$. Computations on Magma \cite{Magma} suggests that the statement of Conjecture \ref{conjecture} holds for any pair of square-free positive integers $(A,B)$ such that $\text{max}\{A, B\} < 100$. 
\end{remark}

\end{document}